\documentclass[]{amsart}
\usepackage[left=2.7cm,right=2.7cm,top=3.5cm,bottom=3cm]{geometry}
\usepackage[english]{babel}
\usepackage{pifont}
\usepackage[latin1]{inputenc}
\usepackage[T1]{fontenc}
\usepackage{amsmath}
\usepackage{amsfonts}
\usepackage{amssymb}
\usepackage{MnSymbol}
\usepackage{amsthm}
\usepackage{amscd}
\usepackage{upgreek}
\usepackage[all]{xy}
\usepackage{graphicx}
\usepackage[usenames,dvipsnames]{color}
\usepackage{mathrsfs}
\usepackage{caption}
\usepackage{tikz-cd} 
\usepackage{braket}
\usepackage{mathrsfs}
\usepackage[colorlinks = true,
linkcolor = black,
urlcolor = black,
bookmarksopen = true, 
backref = page]{hyperref}

\theoremstyle{plain}
\newtheorem{thm}{Theorem}[section]
\newtheorem{cor}[thm]{Corollary}
\newtheorem{lem}[thm]{Lemma}

\theoremstyle{definition}

\newtheorem{rmk}[thm]{Remark}
\newtheorem{rmks}[thm]{Remarks}

\renewcommand{\ge}{\geqslant}
\renewcommand{\le}{\leslant}

\newcommand{\field}[1]{\mathbb{#1}}
\newcommand{\Q}{\field{Q}}
\newcommand{\C}{\field{C}}

\newcommand{\R}{\field{R}}

\newcommand{\Z}{\field{Z}}
\newcommand{\A}{\field{A}}
\newcommand{\F}{\field{F}}

\newcommand{\G}{\field{G}}

\newcommand{\T}{\field{T}}

\DeclareMathOperator{\End}{End}

\DeclareMathOperator{\Gal}{Gal}

\DeclareMathOperator{\Spec}{Spec}

\DeclareMathOperator{\disc}{disc}

\DeclareMathOperator{\gl}{GL}
\DeclareMathOperator{\pgl}{PGL}

\newcommand{\calo}{\mathscr O}

\newcommand{\cals}{\mathscr S}
\newcommand{\scal}{\mathcal S}

\newcommand{\gotn}{\mathfrak n}

\renewcommand{\ge}{\geqslant}
\renewcommand{\le}{\leqslant}

\newcommand{\interior}[1]{%
	{\kern0pt#1}^{\mathrm{o}}%
}
\newtheorem{innercustomgeneric}{\customgenericname}
\providecommand{\customgenericname}{}
\newcommand{\newcustomtheorem}[2]{%
	\newenvironment{#1}[1]
	{%
		\renewcommand\customgenericname{#2}%
		\renewcommand\theinnercustomgeneric{##1}%
		\innercustomgeneric
	}
	{\endinnercustomgeneric}
}

\newcustomtheorem{customthm}{Theorem}

\numberwithin{equation}{section}

\title{A note on equidistribution on a product of  Shimura curves and  Andr\'e--Oort }
\author{Francesco Maria Saettone}
\address{Department of Mathematics, Weizmann Institute of Science, Israel}
\email{francesco.saettone@weizmann.ac.il}

\makeindex

\begin{document}
	
    \begin{abstract}
        In this short note we show that Galois orbits of CM points equidistribute on a product of $r\ge 2$ non-isomorphic Shimura curves by applying the adelic toral-packet equidistribution theorem of Aka--Luethi--Michel--Wieser. As a consequence, we deduce Andr\'e--Oort for the product of those curves, previously studied by Edixhoven and Yafaev, replacing GRH by a Linnik-type splitting condition at two auxiliary primes.
    \end{abstract}

\maketitle

\tableofcontents
    
\section{Introduction}

Nowadays, despite being commonly referred to as a {\em conjecture}, the Andr\'e--Oort conjecture is a theorem by the work of many (see \cite{tsi} for a recent survey), proved unconditionally via the Pila--Zannier strategy, which was first introduced for a new proof of Manin--Mumford in \cite{pz} and then exploited by Pila for the first unconditional proof of Andr\'e--Oort for an arbitrary product of modular curves in \cite{pila}. Nonetheless, other conditional results in the case of product of two modular curves were obtained earlier, under GRH, by Edixhoven in \cite{edix} and in \cite{edix2}. Under the same hypothesis, the two dimensional case was extended to Shimura curves by Yafaev in \cite{yaf}.
\medskip

In this modest note we consider the case of $r\ge 2$ {\em non}-isomorphic Shimura curves $\scal_1,\dots,\scal_r$, and, bye an equidistribution result for special points on their product 
\[
\scal_1\times\dots\times\scal_r
\]
we obtain as a corollary the corresponding case of the Andr\'e--Oort conjecture dropping GRH and assuming instead a mild congruence condition on two auxiliary primes (also known as a Linnik-type condition). The equidistribution input is the theorem of Aka--Luethi--Michel--Wieser \cite[Theorem~1.8]{almw}, which in turn relies on Einsiedler--Lindenstrauss' joining theorem \cite{el} together with some character-spectrum analysis. 

Ergodically, the situation we consider is especially simple: since the curves $\scal_1,\dots,\scal_r$ are pairwise non-isomorphic, the product $\scal_1\times\cdots\times \scal_r$ contains no graphs of Hecke correspondences. Consequently, the only positive-dimensional special subvarieties are fibers of CM points, and all joinings are trivial. 
A similar consideration on disjointness plays a crucial role in \cite{almw}, where the authors prove a simultaneous equidistribution on the supersingular locus of modular curves. Other simultaneous equidistributions were proved unconditionally in \cite{cv} and \cite{fms} using Ratner's theorem\footnote{At the price of fixing the discriminant and letting  the conductor varying $p$-adically.}. 

\medskip

Whenever joinings other than the trivial one appear, the situation becomes radically subtler. The majestic work of Khayutin \cite{kha} on the mixing conjecture exploited a striking combination of ergodic and analytic number theory especially to deal with the non-trivial joinings coming from  Hecke correspondences, and it is conditional to assuming no exceptional Landau--Siegel zero, plus a Linnik-type condition. The latter was removed, by assuming GRH, in the recent work of Blomer--Brumley--Khayutin \cite{bbk} for $r=2$. 
\\

For $r\ge2$, let $G_1,\dots,G_r$ be $r$ non-isomorphic forms of $\pgl_2$  and let $\scal_1,\dots,\scal_r$ be the  Shimura curves relative to $G_1,\dots,G_r$.

\begin{thm}\label{t:1}
 Consider a sequence of CM points $(z_n)_n$ in $\scal_1\times\dots\times\scal_r$ such that for all $n$, the coordinates of $z_n$ have the same CM field $E_n$. Fix two odd distinct primes $p,q$ and assume that, for all natural numbers $n$, the primes $p,q$ split in $E_n$.
 If the sequence $(z_n)_n$ is strict\footnote{I.e., its intersection with any proper special subvariety is finite.}, then its Galois orbits become equidistributed in $\scal_1\times\dots\times\scal_r$ as the discriminants of $E_n$ tend to infinity.
\end{thm}
The above theorem yields, under the splitting condition on the $E_n$'s, a special case of a general (and, to the best of the author's knowledge, still broadly open) folklore conjecture first written down by Shou-Wu Zhang in \cite[Conjecture 2.2]{swz}. It is not hard to see that the equidistribution of Galois orbits of CM points is a refinement, and actually implies, the Andr\'e--Oort conjecture (see, for instance, \cite[Remarks.(1), p.3660]{swz}).

For $r=2$, it seems plausible that the techniques of \cite[Theorem~2.2]{bbk} could be adapted to remove the Linnik-type condition and to make the equidistribution effective, since the absence of Hecke graphs should allow one to drop several of their hypotheses.

\medskip

In any case, Theorem~\ref{t:1} implies the following (conditional) case of the Andr\'e--Oort conjecture.
\begin{cor}\label{ao}
Fix two odd distinct primes $p$ and $q$. Let $Z \subset \scal_1\times\dots\times \scal_r$ be a closed irreducible subvariety containing a Zariski dense set of CM points such that $p$ and $q$ split in all their CM fields $E_n$. Then $Z$ is of the form 
\[
\prod_{i\in I}\scal_i \times \prod_{j\notin I}\{x_j\}
\]
where $I\subset\{1,\dots,r\}$ and $x_j\in\scal_j$ are CM points.
\end{cor}

\begin{rmks}
 \begin{itemize}
    \item[]
\end{itemize}
\noindent{\bf 1.}
In the case  of a product of two such Shimura curves $\scal_1\times\scal_2$, which falls within the cases studied in \cite{yaf}, any positive-dimensional (special) subvariety $Z$ is a curve containing infinitely many CM points $z_n=(x_n,y_n)$.  By \cite[Proposition~3.1]{yaf}, for all but finitely many $n$, the CM fields of $x_n$ and $y_n$ are isomorphic. Consequently, the hypotheses of Theorem~\ref{t:1} and Corollary~\ref{ao} may be weakened in this case.
\\
\noindent{\bf 2.} 
Let $Z$ be a closed irreducible subvariety of $\C^n$ containing a Zariski-dense set $\Sigma$ of special points. Rather interestingly, by \cite[Theorem 8.1]{edix2}, either by assuming GRH for imaginary quadratic fields or by assuming that the set of CM points can be taken in a single isogeny class (naturally, this last assumption does not apply in our scenario, as there are no isogenies at all), there is a special curve $C\subset Z$ such that for all but finitely many $x\in\Sigma$, one has $x\in C$. The proof of \cite[Theorem 8.1]{edix2} seems to be straightforwardly adaptable to a product of Shimura curves. Therefore, by \cite[Proposition 3.1]{yaf}, one would obtain that the coordinate of almost all special points $x$ have the same CM field.

\end{rmks}
\medskip
\noindent

One final comment on the literature of equidistribution on product of Shimura curves and the various hypothesis there exploited. From a broader perspective, several complementary approaches yield 
equidistribution of Galois orbits of CM points on products of Shimura curves under 
different sets of significant hypotheses: 
\begin{itemize}
\item Khayutin's work \cite{kha} treats self-products of Shimura curves under the assumption of no Siegel zeros and the Linnik-type condition; 

\item the adelic equidistribution theorem of Aka--Luethi--Michel--Wieser \cite[Theorem~1.8]{almw}, based on Einsiedler--Lindenstrauss's joining theorem \cite{el}, applies to the finite-level packets used here;

\item Blomer--Brumley--Khayutin's work \cite{bbk} handles products of two Shimura curves curves under GRH; 

\item Blomer--Brumley--Radziwi\l{}\l{} \cite{bbr} prove a joint Linnik
equidistribution theorem for distinct quaternionic varieties at almost maximal level. Their result applies when two quaternion algebras are non-isomorphic and replaces GRH by a zero-free region assumption for the quadratic Dirichlet $L$-function.
\end{itemize}

This note addresses one specific case within the broader landscape described above and may be viewed as an equidistribution analogue, under weaker hypotheses, of a portion of Yafaev's result in \cite{yaf}. The main input is the finite-level toral-packet equidistribution theorem of \cite{almw}, whose proof is ultimately based on the joining theorem of Einsiedler--Lindenstrauss \cite{el} and does not require any form of GRH.

\subsubsection*{Acknowledgments}
I thank Farrell Brumley for discussions related to this note, Harry Schmidt for encouragement  to write it down, and George Papas and the anonymous referee for helpful comments. I was supported by ISF grant 2067/23 and by the ERC, SharpOS, 101087910.

\section{Shimura curves and Equidistribution}

\subsection{Quaternion algebras and special points}

\subsubsection{Indefinite quaternion algebras}

Consider an {\em indefinite} quaternion algebra $B$ over $\Q$, which means that $B\otimes \R\simeq M_2(\R)$. Denote by $d$ its discriminant, i.e., the product of the finitely many primes $p$ such that $B\otimes\Q_p$ is not isomorphic to $M_2(\Q_p)$.

Let $G$ be the reductive group over $\Q$ whose functor of points sends a commutative $\Q$-algebra $A$ to
\[
G(A)=(B\otimes_\Q A)^\times/A^\times.
\]
Consider the real embedding $h_0\colon \text{Res}_{\C/\R}\G_m\rightarrow G_\R$. 
We then fix an isomorphism $B\otimes\R\simeq M_2(\R)$. This isomorphism induces a map $B^\times\rightarrow \gl_2(\R)$ which gives an action of $B^\times$ on the conjugacy class of $h_0$, which is isomorphic to $\mathscr{H}=\C-\R$, by M\"obius transformations. 

Henceforth, let $\A$ and $\A_f$ denote the $\Q$-adeles and the finite $\Q$-adeles respectively.
For any open subgroup $U$ of $G(\A_f)$ which is compact modulo $\widehat{\Z}^\times$, we consider
\begin{equation}\label{anShim}
\scal_U^{\text{an}}:=G(\Q)\backslash \mathscr{H}\times G(\A_f)/U\cup \{\text{cusps}\}
\end{equation}
whose canonical model is the Shimura curve $\scal_U$ over $\Q$. It is a proper and smooth curve over its reflex field $F$. Note that the set $\{\text{cusps}\}$ is non-empty if and only if $B=M_2(\Q)$; equivalently, $\scal_U^{\text{an}}$ is compact if and only if $B$ is a division algebra.

We briefly recall how to define in generality a Shimura variety via a pair consisting of a reductive group and of a symmetric domain. A \emph{Shimura datum} is a pair $(G, X)$ where $G$ is a reductive algebraic group over $\mathbb{Q}$, and $X$ is a $G(\R)$-conjugacy class of homomorphisms $h\colon \text{Res}_{\C/\R}\G_m \to G_{\mathbb{R}}$ satisfying  Deligne's axioms (see \cite{del}).

Given a compact open subgroup $K \subset G(\mathbb{A}_f)$, the associated \emph{Shimura variety} is given by
\[
\mathrm{Sh}_K(G,X) = G(\mathbb{Q}) \backslash (X \times G(\mathbb{A}_f)/K).
\]

Let $(H, X_H) \subset (G, X)$ be a Shimura subdatum. Then for $K_H = K \cap H(\mathbb{A}_f)$, the image of the morphism $\mathrm{Sh}_{K_H}(H, X_H) \to \mathrm{Sh}_K(G, X)$
is called a \emph{special} subvariety.

\subsubsection{Special points}
Let $E$ be a quadratic imaginary extension of $\Q$ and fix an embedding $\rho\colon E\hookrightarrow B$. The scheme of $\mathit{CM}$ $\mathit{points}$ by $E$ consists of the points $z$ of $\scal_U$ that can be represented by $(z_0,g)\in \mathscr{H}\times G(\A_f)$ via (\ref{anShim}), where $z_0$ is the unique point fixed by $E^\times$. By the work of Shimura, it is a finite subscheme of $\scal_U$ defined over $E^{\text{ab}}$. By taking the union of $\scal_U^{\tau(E^\times)}$ over all pairs $(E,\tau)$, we obtain the CM ind-subscheme $\scal_U^{\text{CM}}$. The absolute Galois group of $E$, which we denote by $G_E$, acts on $\scal_U^{\text{CM}}$ via
\[
\sigma.(x_0,g)=(x_0,\text{rec}_E(\sigma)g)
\]
where $\text{rec}_E$ is Artin's reciprocity map. If we consider CM points of conductor $c$, this action factors through $\text{Gal}(H[c],E)$, where $H[c]$ is the ring class field of $E$ of conductor $c$.
\\

Consider  an order\footnote{See as in \cite[Sec.1.5.1]{swz2} for the detailed construction.} $R$ of $B$ of discriminant $\gotn$ which contains $\rho(\calo_E)$ and the corresponding  Shimura curve of level $\widehat{R}^\times$. Then for $z$ a CM point by $E$ we consider 
\begin{equation}\label{e:order}
\End(z):=g\widehat{R}^\times g^{-1}\cap \rho(E)
\end{equation}
which is an order in $E=\rho(E)$ independent of the choice of $g\in G(\A_f)$. The $\mathit{conductor}$ of $z$ is defined as the unique integer $c$ such that
\[
\End(z)=\Z + c\calo_E.
\]
The discriminant of $\End(z)$ is of the form $Dc^2$, where $D$ is the discriminant of $E$  and $c$ the conductor.

Let us recall the following characterization of  CM points in terms of an adelic double quotient. Let $T$ be the $\Q$-rational torus in $G$. The set of CM points in $\scal_U$ is then in bijection with
\[
T(\Q)\backslash G(\A_f)/U.
\]

\subsubsection{Moduli interpretation}
From now on we will assume that all the Shimura curves we consider have maximal level structure, since this is sufficient for the Andr\'e--Oort application considered here. This assumption also verifies automatically the norm-surjectivity hypothesis in the  equidistribution theorem \cite[Theorem~1.8]{almw}; see Theorem~\ref{t:almw-specialized} below. Similarly, the assumption that the Shimura curves are defined over $\Q$ rather than over a totally real number field is irrelevant both to the diophantine and to the ergodic setting. 

We now briefly recall the moduli interpretation of Shimura curves over $\Q$. We invite the reader to refer to \cite{buz} and \cite{swz2}.

Let $b\mapsto \overline{b}$ denote the canonical involution of $B$, and let $b^*=t^{-1}\overline{b}t$, where $t$ is such that $t^2=d$, where $d$ is the discriminant of $B$. Let also $\calo_B$ denote the maximal order of $B$. The curve $\scal$ is the coarse moduli scheme for the following moduli problem on $\Q$-schemes:
\[
S\mapsto [A,i,\theta]
\]
where the isomorphism classes consist of:
\begin{itemize}
\item   an abelian scheme $A$ over $S$ of relative dimension $2$; 
\item an injective ring morphism $i\colon\calo_B\rightarrow \End_S(A)$;
\item a principal polarization $\theta$ such that for every geometric point $s$, the Rosati involution on $\End(A_s)$ induces the involution $*$.
\end{itemize}
As the principal polarization $\theta$ exists, and is unique, we will safely ignore it. The couple $(A,i)$ is usually called a {\em QM abelian surface} or {\em false elliptic curve}.
By extending the moduli interpretation to the integers, one obtains that the Shimura curve has a proper\footnote{In the case of a modular curve, finitely many cusps are needed for properness.} regular integral model $\cals$ over $\Z$.

\subsubsection{Special subvarieties}

Let $\scal_1,\dots,\scal_r$ be quaternionic Shimura curves associated with quaternion algebras $B_1,\dots,B_r$ over $\mathbb{Q}$ and with full level structure. Henceforth, we shall make the following assumption:
\[
B_i \not\simeq B_j
\] 
for all $i\neq j$.

For $1\le i\le r$, let $G_i$ be the reductive group defined by the functor of points $G_i(A)=(B_i\otimes_{\mathbb{Q}}A)^\times/A^\times$, where $A$ is a commutative $\mathbb{Q}$-algebra. Then the fiber product  
$$
 \scal_1\times\dots\times \scal_r
$$ 
is obviously itself a Shimura variety over $\Q$ associated with the datum $\big(\prod_{i=1}^r G_i, (\C-\R)^r\big)$.

In what follows we spell out the proof of the following Lemma, probably well-known to some experts, which characterizes the special subvarieties of our product of non-isomorphic Shimura curves.

\begin{lem}
Let $I\subset \{1,\dots,r\}$. The positive-dimensional special subvarieties of $ \scal_1\times\dots\times \scal_r$ are those of the form
\[
\prod_{i\in I}\scal_i
\times \prod_{j\notin I} \{x_j\},
\]
where each $x_j \in \scal_j$ is a CM point.  
\end{lem}

\begin{proof}
A special subvariety of $X$ corresponds to a Shimura subdatum
\[
(H,X_H)\subset
(G_1\times\dots \times G_r,\;
(\C-\R)^r).
\]
By \cite[Theorem 4.3]{moo}, every special subvariety contains a special point,
hence the  center of $H$ is a torus and $H$ is generated by the centers of its simple factors.

Since the quaternion algebras $B_i$ are pairwise non-isomorphic, the adjoint 
groups 
are pairwise non $\Q$-isogenous.  Therefore, there is no nontrivial algebraic subgroup of
$B_i^\times \times B_j^\times$ surjecting onto both factors for $i\neq j$.
Consequently, $H$ cannot contain any mixed simple factor linking two different
indices. It follows that every possible $H$ is a product of the form
\[
H \;=\;
\prod_{i\in I} B_i^\times 
\;\times\;
\prod_{j\notin I} T_j,
\]
where $I\subset\{1,\dots,r\}$ and each $T_j\subset B_j^\times$ is a torus. The corresponding special subvariety is then
\[
\prod_{i\in I} \scal_i
\;\times\;
\prod_{j\notin I} \{x_j\}
,
\]
with $x_j$ the CM point determined by $T_j$.  
\end{proof}



The next Lemma immediately implies the trivial cases of Corollary \ref{ao} for $r=2$. 

\begin{lem}Let $X$ and $X'$ be smooth projective geometrically irreducible curves over a
field $F$, and let
$p_1 : X \times X' \longrightarrow X$, and $p_2 : X \times X' \longrightarrow X'$
be the projections on the first and second factor respectively.  
Let $C \subset X \times X'$ be an irreducible curve
such that $p_1(C)=\{x_0\}$ for some closed point $x_0 \in X(F)$, and assume that
$p_2|_C$ is non constant. Then $C=\{x_0\}\times X'$.
\end{lem}
\begin{proof}
Set $C' := p_1^{-1}(x_0)=\{x_0\} \times S'$ and let $k(x_0)$ denote the residue field at $x_0$.
Then $C'$ is an integral projective curve isomorphic to $X' \times_{\text{Spec}F}\text{Spec}(k(x_0))$. Consider $i_C : C \hookrightarrow S\times X'$.
Since $p_1 \circ i_C$ is constant, $\iota_C$ factors uniquely through $C'$, giving a morphism $\varphi : C \rightarrow C'$.
Because $p_2|_C$ is nonconstant, the morphism $\varphi$ is nonconstant.
Hence it induces an inclusion at the level of  function fields $F(C') \hookrightarrow F(C)$.
Both $F(C')$ and $F(C)$ have transcendence degree $1$ over $F$, so
the field extension $F(C)/F(C')$ is finite. Thus $\varphi$ is
a finite morphism between integral projective curves.
We also have that $i_C$ is a closed immersion, so $\varphi$ is a closed immersion as well.
A finite morphism that is also a closed immersion between integral curves
is an isomorphism, so $C \simeq  \{x_0\}\times X'$
as closed subschemes of $X \times X'$.
\end{proof}

Therefore, if one takes $r=2$ and $X=S_1$, $X'=S_2$ and assumes that $C$ contains at least one CM point (and we actually assume it contains infinitely many of them), then $C$ has to be a special subvariety.

\begin{rmks}
    \begin{itemize} \item[]\end{itemize}
    \noindent{\bf 1.} We point out that, in Yafaev's paper \cite{yaf} (as well as in \cite{edix2}), effective Chebotarev, which indeed requires GRH, is still necessary even for the (simpler) case $B\not\simeq B'$, as it is crucially exploited in \cite[Lemma 4.7]{yaf}. The proof is still simplified by such the non-isomorphic hypothesis, as the automorphism $\sigma$ does not exist, then one of the projections is forced to be constant. Nonetheless it does not seem possible to avoid GRH.
    \\
    \noindent{\bf 2.} As a (very) minor remark, we notice that the above Lemma shows that the first line after the statement of \cite[Theorem 4.1]{yaf}, although correct, is independent of the hypothesis of the cited Theorem.
    
\end{rmks}

\subsection{Adelic quotients and equidistribution}

Given a linear algebraic group $G$ over $\Q$, its adelic quotient is denoted as $[G]=G(\Q)\backslash G(\A)$. If $G$ is anisotropic, then the locally compact space $[G]$ carries a unique $G(\A)$-invariant probability measure, which we dub simply as the Haar measure on $[G]$.

\subsubsection{Toral packets}
For every $i\in\{1,\dots,r\}$, let $\iota_i$ be an embedding of a CM field $E$ into $B_i$, which induces the morphism of algebraic groups
$\iota_i\colon T_E=(\text{Res}_{E/\Q}\G_m)/\G_m\rightarrow G_i$. Let 
\[
\G:=G_1\times\dots\times G_r\ .
\]
Then the $\iota_i$'s induce a diagonal morphism $\iota\colon T_E\rightarrow \G$. We denote the image of the latter morphism by $\T_\iota$. For $g=(g_1,\dots,g_r)\in\G(\A)$, the associated compact {\em toral packet} is given by the right translate of $[\T_\iota]$ by $g$, i.e.,
\[
[\T_\iota g]=\T_\iota(\Q)\backslash \T_\iota(\A)g\subset [\G]\;.
\]
We equip a toral packet with the pushforward of the natural probability Haar measure on $[\T_{\iota}]$. If $K_f\subset \G(\A_f)$ is a compact open subgroup, we write $[\G]_{K_f}:=[\G]/K_f$ and $[\T_{\iota}g]_{K_f}$ for the  finite-level quotient, and the projection of the packet to it.

Let us recall the notion of {\em discriminant} for a toral packet following \cite[Definition 2.2]{kha}, which is equivalent to the one introduced in \cite[Section 6.1]{elmv}. For a rational prime $p$ and $z\in\scal^{\text{CM}}$ with CM by $E$, define the discriminant $D_p$ as the discriminant of the local order $\End(z)\otimes\Z_p$ (as constructed in (\ref{e:order}), where we take $R$ to be maximal in $B$). Note that, for all $p$ such that $E_p$ is ramified, we have $D_p=1$. 
Consider the order in $E$ given by
$$\Lambda=\bigcap_p\End(z)\otimes\Z_p$$
which is a finite index  $\Z$-sublattice of the ring of integers $\calo_E$. Indeed, its discriminant amounts to $\prod_pD_p$. Therefore any homogeneous toral quotient $[T_E(\A)g]$ has an associated lattice $\Lambda\subset\calo_E$.
\\The discriminant of a toral packet $[\T_\iota g]$ is thus defined as
\[
\text{disc}([\T_\iota g]):=\min_{1\le i\le r}\text{disc}([\T_{\iota_i}g_i]),
\]
where $[\T_{\iota_i}g_i]$ denotes the $i$-th projection of the toral packet and $\text{disc}([\T_{\iota_i}g_i])$ is the discriminant of the associated lattice.

\subsubsection{Ergodic equidistribution}

The equidistribution statement needed in this note is a direct specialization of Aka--Luethi--Michel--Wieser \cite[Theorem~1.8]{almw}. We record it in the notation used here.

\begin{thm}\label{t:almw-specialized}
 For each $i$, let $R_i\subset B_i$ be a maximal order and put
\[
K_{f,i}:=\operatorname{im}\big(\widehat{R_i}^{\,\times}\rightarrow G_i(\A_f)\big),
\qquad
K_f:=\prod_{i=1}^rK_{f,i}\subset \G(\A_f).
\]
Let $p$ and $q$ be two distinct odd rational primes. Let $(E_n)_n$ be a sequence of quadratic fields such that $p$ and $q$ split in $E_n$ for every $n$. Let $\iota_n=(\iota_{1,n},\dots,\iota_{r,n})$ and $\iota_{i,n}:E_n\hookrightarrow B_i$,
and let the induced diagonal torus be denoted by $\T_{\iota_n}\subset \G$. For any $n$, let $g_n\in \G(\A)$. If
\[
\disc([\T_{\iota_n}g_n])\rightarrow \infty,
\]
then the packets $[\T_{\iota_n}g_n]_{K_f}$ equidistribute in $[\G]_{K_f}$.
\end{thm}

\begin{proof}
This is \cite[Theorem~1.8]{almw}, specialized to the compact open subgroup $K_f$ above. We recall the two points relevant to the present application. First, \cite[Theorem~1.8]{almw} is stated for arbitrary distinct rational quaternion algebras\footnote{In particular, no definiteness assumption is imposed.}. Secondly, the finite-level hypothesis in \cite[Theorem~1.8]{almw} is the surjectivity of the reduced norm
\[
\text{Nr}_{B_i}\colon K_{f,i}\twoheadrightarrow
\widehat{\Z}^{\times}/(\widehat{\Z}^{\times})^2
\]
for every $i$. Since we work at maximal level, $K_{f,i}$ is the image of $\widehat{R_i}^{\,\times}$ with $R_i$ maximal. In particular $R_i$ is an Eichler order, and \cite[Remark~1.10 and Lemma~9.5]{almw} give
\[
\text{Nr}_{B_i}(\widehat{R_i}^{\,\times})=\widehat{\Z}^{\times}.
\]
After passage to $PB_i^\times$, the reduced norm is defined modulo squares; hence the required norm surjectivity condition holds for $K_{f,i}$. 
\end{proof}



\subsubsection{Proof of Theorem~\ref{t:1}}

Let
\[
   z_n=(z_{1,n},\ldots,z_{r,n})
   \in \scal_1\times\cdots\times\scal_r
\]
be as in the statement of the theorem. Choose embeddings $\iota_{i,n}:E_n\hookrightarrow B_i$ and adelic representatives $g_{i,n}\in G_i(\A)$ for the corresponding CM
points. Set $\iota_n=(\iota_{1,n},\ldots,\iota_{r,n})$, and $g_n=(g_{1,n},\ldots,g_{r,n})\in \G(\A)$, and let $\T_{\iota_n}\subset \G$ be the diagonally embedded torus associated
with $E_n$.

Let $K_\infty\subset \G(\R)$ be the product of the archimedean maximal compact subgroups, so that
\[
\scal_1(\C)\times\cdots\times\scal_r(\C)\simeq\G(\Q)\backslash \G(\A)/(K_\infty K_f)\ .
\]
Denote by $\pi_\infty\colon [\G]_{K_f}\rightarrow \G(\Q)\backslash \G(\A)/(K_\infty K_f)$ the natural projection.
By Shimura reciprocity, the Galois orbit of $z_n$ is the image, under
$\pi_\infty$, of the finite-level toral packet
\[
   [\T_{\iota_n}g_n]_{K_f}\subset [\G]_{K_f}.
\]
Equivalently, the normalized counting measure on
$\text{Gal}_{E_n}.z_n$ is the pushforward under
$\pi_\infty$ of the Haar probability measure on this packet.

It remains only to check that the hypotheses of
Theorem~\ref{t:almw-specialized} apply to these packets. The primes $p$ and
$q$ split in $E_n$ by assumption. For the discriminant condition, set
$\calo_{i,n}:=\End(z_{i,n})\subset E_n$.
At the fixed maximal finite level, the packet-discriminant comparison gives a constant $C_i>1$, depending only on $B_i$ and on the level, such that $C_i^{-1}|\disc(\calo_{i,n})|\le\disc([\T_{\iota_{i,n}}g_{i,n}])\le C_i|\disc(\calo_{i,n})|$.
We use this comparison in the form recorded in \cite[Section~7.2.3]{almw}. If
$c_{i,n}$ denotes the conductor of $\calo_{i,n}$, then $|\disc(\calo_{i,n})|=|\disc(E_n)|\,c_{i,n}^{2}$.
Thus every projected packet discriminant tends to infinity as
$|\disc(E_n)|\to\infty$. Hence
\[
   \disc([\T_{\iota_n}g_n])
   =
   \min_i \disc([\T_{\iota_{i,n}}g_{i,n}])
   \to \infty .
\]

By Theorem~\ref{t:almw-specialized}, the probability measures on $[\T_{\iota_n}g_n]_{K_f}$ weak-$*$ converge to Haar measure on $[\G]_{K_f}$. The projection $\pi_\infty$ is continuous and proper, since it is obtained
by quotienting on the right by the compact group $K_\infty$. Hence, for every compactly supported continuous test function on $\G(\Q)\backslash \G(\A)/(K_\infty K_f)$, its pullback to $[\G]_{K_f}$ is again compactly supported and continuous. Therefore the weak-$*$ convergence given by Theorem~\ref{t:almw-specialized} descends after pushing forward by $\pi_\infty$. Via the Shimura reciprocity as above, the pushed-forward packet measures are precisely the normalized counting measures on $\Gal_{E_n}.z_n$, while the pushforward of Haar measure is the product Haar probability measure on $\scal_1(\C)\times\cdots\times\scal_r(\C)$. Hence these Galois orbit measures weak-$*$ converge to the product Haar measure, proving Theorem~\ref{t:1}.

\begin{rmk}
    Using the language of adelic double quotients and the equidistribution theorem above, it seems possible to prove the equidistribution of the reduction of Galois orbits of special points to the special fiber $(\cals_1\times\dots\times\cals_r)\times\Spec(\overline{\F}_p)$. Naturally, the behaviour of $B$ at $p$ influences deeply the geometry of the special fiber at $p$. For individual equidistribution results covering both the case where $B_p$ is unramified and the case where $B_p$ is ramified, we refer to \cite{fms2} and its citation orbit.
\end{rmk}

\end{document}